\documentclass[a4paper]{amsart}

\usepackage[english]{babel}
\usepackage[utf8x]{inputenc}
\usepackage[T1]{fontenc}

\usepackage[a4paper,top=3cm,bottom=2cm,left=3cm,right=3cm,marginparwidth=1.75cm]{geometry}

\usepackage{amsmath}
\usepackage{amsthm}
\usepackage{hyperref}
\usepackage{amssymb}
\usepackage{graphicx}
\usepackage{amssymb}
\usepackage{epstopdf}
\usepackage{enumitem}

\theoremstyle{definition}
\newtheorem{theorem}{Theorem}[section]
\newtheorem{remark}[theorem]{Remark}
\newtheorem{lemma}[theorem]{Lemma}
\newtheorem{definition}[theorem]{Definition}
\newtheorem{example}[theorem]{Example}
\newtheorem{proposition}[theorem]{Proposition}

\numberwithin{equation}{section}

\DeclareMathOperator{\R}{\mathbb{R}} 

\begin{document}

\title[Embedding in Euclidean space of 
coherent configuration of type (2,2;3)]{On the two-distance embedding in real Euclidean space of 
coherent configuration of type (2,2;3)}
\author{Eiichi Bannai, Etsuko Bannai, Chin-Yen Lee, Ziqing Xiang, Wei-Hsuan Yu}

\begin{abstract}

Finding the maximum cardinality of a $2$-distance set in Euclidean space is a classical problem in geometry. Lison\v ek in 1997 constructed a maximum 
$2$-distance set in $\mathbb R^8$ with $45$ points. 
That $2$-distance set constructed by Lison\v ek has a distinguished structure of a coherent 
configuration of type $(2,2;3)$ and is embedded in two concentric 
spheres in $\mathbb R^8$. In this paper we study whether there exists 
any other similar embedding of a coherent configuration of type $(2,2;3)$ as a $2$-distance set in $\mathbb R^n$, without assuming any restriction on the size of the set. We prove that there exists no such example other than that of Lison\v ek. The key ideas of our proof are as follows: (i) study 
the geometry of the embedding of the coherent configuration in Euclidean 
spaces and to drive diophantine equations coming from this embedding.  (ii) solve diophantine equations with certain additional conditions of integrality of some parameters of the combinatorial structure by using the method of auxiliary equations.
\end{abstract}
\maketitle
\section{Introduction}

Let $X,Y$ be finite subsets of $\mathbb R^d$, we define $A(X,Y)=\{\|x-y\|: x\in X,y\in Y,\text{ and }x\neq y\}$.
A finite set $X\subset\mathbb R^d$ is called  an \emph{$s$-distance set} if $|A(X,X)|=s$. By assigning a set of linearly independent polynomials to an $s$-distance set, Bannai, Bannai and Stanton \cite{bannai1983upper}  and Blokhuis \cite{blokhuis1983few} gave an upper bound for the cardinality of an $s$-distance sets $X$ in $\mathbb R^d$,
\begin{align*}
    |X|\leq \binom{d+s}{s}.
\end{align*}

Sz{\"o}ll{\H{o}}si and {\"O}sterg{\aa}rd \cite{szollHosi2018constructions} had the latest progress for the construction of $s$-distance sets for $s\leq 6, d\leq 8$.
They gave an algorithm to exhaust the $s$-distance sets in the space of small dimension and small $s$ at least $3$.
Lison\v ek \cite{lisonvek1997new} classified all two-distance sets for the dimension $d=4,5,6$ and $7$. For dimension $8$, he gave an example which attains the maximum cardinality  $\binom{8+2}{2}=45$. However the classification problem is still open. The example of Lison\v ek, given below, is the only known $s$-distance set that  attain the Bannai-Bannai-Stanton and Blokhuis bound $\binom{d+s}{s}$. In fact, these 45 points are distributed in a very symmetrical manner. It can be divided in two parts : one part is 9 points forming a regular simplex, and the other part is $36$ points coming from the spherical embedding of a Johnson scheme $J(9,2)$.

\begin{example}(Lison\v ek) \label{example}
Let $X_1=\{    -e_i+\frac 1 3\sum_{k=1}^9e_k : 1\leq i\leq 9\}$ and $X_2= \{e_i+e_j : 1\leq i<j\leq 9\}$, where $\{e_i:1\leq i\leq 9\}$
 are the standard orthonormal basis of $\mathbb R^9$. All the points are on the hyperplane $H=\{x=\sum_{i=1}^9x_ie_i:\sum_{i=1}^9x_i=2\}$, and its affine dimension is $8$.
Then, $X_1\cup X_2$ is a maximum two-distance set in $\mathbb R^8$. Notice that the distinct two distances in $X_1 \cup X_2$ are $\sqrt 2$ and $2$. The radius of $X_1$ and $X_2$ are $\frac{2}{\sqrt 3}$ and $\sqrt 2$ respectively. 
\end{example}

In fact, $X_2\subset \mathbb R^d$ can be interpreted as a block design with the underlying set $X_1$. The union $X_1\cup X_2$ has the elegant structure:
 \begin{enumerate}
\item $  X_1$ is the regular simplex in $\mathbb R^{d}$;
\item $  X_2$ is a scaling of the spherical embedding of a strongly regular graph in $\mathbb R^{d}$; 
\item distance between point and block only depends on whether the point in block or not;
\item $|A(  X_1,   X_1)|=1$ and $|A(  X_1,  X_2)|=|A(  X_2,  X_2)|=2$.
\end{enumerate}
Condition (4) is a consequence of (1)-(3), and we list here for convenience.

A coherent configuration of type (2,2;3) is a combinatorial structure and it consists of a point set $V$, a block set $B$ and a set of finite number of relations $\{R_i\}$.
The block design in Example \ref{example} is quasi-symmetric, and it is equivalent to the type (2,2;3) coherent configuration.
We prove that such coherent configurations can always be embedded into a Euclidean space with the above structure. 
\begin{theorem}\label{embcond} Let $(V\cup B, \{R_1,\dots,R_9\})$ be a coherent configuration of type (2,2;3), where $V$ is a finite set. Then there is a map $i:V\cup B\to\mathbb R^{d}$, $d=|V|-1$,  such that  $X_1=i(V)$ and $X_2=i(B)$ satisfying the conditions (1)-(4) in Example \ref{example}.
\end{theorem}

In this paper we will consider the following problem:
When the embedding of type (2,2;3) coherent configuration in Theorem \ref{embcond} gives a two-distance set in Euclidean space?
Dropping the condition on sizes (i.e., we don't assume $|X_1\cup X_2|=\binom{d+2}{2}$) makes the problem
very broad and difficult. However, to our surprise, we are able to show our main theorem
that there is no other example at all.

\begin{theorem}(Main result) \label{main thm} 
The example given by Lison\v ek is the unique coherent configuration of type (2,2;3) which can be embedded in Euclidean space as a two-distance set satisfying the conditions given in Theorem \ref{embcond}.

\end{theorem}

Nozaki and Shinohara \cite{nozaki2020maximal} study two-distance sets in $\mathbb R^d$
that contain a regular simplex. They solve the case when Larman–Rogers–Seidel (LRS) \cite{larman1977two} ratio is 2. When LRS ratio is 3, they give a partial result by adding some block design structures.
So, our present paper is relevant to  \cite{nozaki2020maximal}.
However, it seems that the situation in our paper is eluded
in their consideration. 

The outline of our paper is as follows. In section 2, we introduce the notions of type (2,2;3) coherent configurations and quasi-symmetric designs and we will prove the Theorem \ref{embcond}. In section 3, we study when the embedding given by Theorem \ref{embcond} is a two-distance set. We derive the conditions for the embedding of a coherent configuration forming a two-distance set into three Diophantine equations $p_1(S,m,x,y)=p_2(S,m,x,y)=p_2(S,m,x,y)=0 $.  In section 4,  we determine the complete solutions of the system of Diophantine equations. In section 5, we prove our main result, Theorem \ref{main thm}.
Please note that we use computers crucially. The results are rigorously proved by applying the notion of real algebraic geometry. We can see the precedents of the basic idea of the proof in Xiang \cite{xiang2018nonexistence}, and Bannai, Bannai, Xiang, Yu and Zhu \cite{bannai2021classification}.

\section{Preliminaries}
We introduce the two well-known concepts, quasi-symmetric designs and  coherent configurations of type (2,2;3). They are proved  as the equivalent notions in Higman \cite{higman1987coherent}. Quasi-symmetric design is a combinatorial object  whose parameters   are non-negative integers. The integral conditions between them  are crucial for classifying them.  
\subsection{Quasi-symmetric designs}
\begin{definition}\cite{cameron1975graph}
A \emph{$t$-design} with parameters $(m,S,\Lambda)$ (or a \emph{$t$-$(m,S,\Lambda)$ design}) is a collection of subsets  $B$  (called \emph{blocks}) of a set $V$ of $m$ points such that 
\begin{enumerate}
    \item every member of $B$ contains $S$ points;
    
    \item any set of $t$ points is contained in exactly $\Lambda$ members of $B$. 
\end{enumerate}
\end{definition}

A $2$-design is called a \emph{quasi-symmetric design} if the cardinality of the intersection of two different blocks takes just two distinct values. These two numbers  are  called the \emph{intersection numbers}. We express these intersection numbers by $\alpha$ and $\beta$, with $\beta<\alpha$. By definition each element $v\in V$ contains in $T$ distinct blocks. For each point $v\in V$ and each block $b\in B$, the following condition is satisfied:
\begin{align*}
    |\{ w\in B:  |b\cap w|=\alpha\text{ and } v\in w \}|=\begin{cases}
    N,&\text{ if }v\in b,\\
    P,&\text{ if }v\not\in b.
    \end{cases}
\end{align*}
 We denote this number by $N$ when $v\in b$ and by $P$ when $v\not \in b$.
The cardinality of quasi-symmetric design is bounded above by ${m\choose 2}$ (see Proposition 3.4 in Cameron and van Lint \cite{cameron1975graph}).
 
\begin{theorem}\cite[Proposition 3.6]{cameron1975graph} \label{thm3} For a $2$-$(m,S,\Lambda)$ design $\mathcal D$ with $4\leq S\leq m-4$, any two of the following imply the third:
\begin{enumerate}
\item $\mathcal D$ is quasi-symmetric;
\item $\mathcal D$ is a $4$-design;
\item $|B|={m\choose 2}$ .
\end{enumerate}

\end{theorem}

The 4-design attaining the lower bound ${m \choose 2}$ is called a \emph{tight} 4-design. Enomoto-Ito-Noda \cite{enomoto1979tight} classified the tight 4-designs.

\begin{theorem}\cite{enomoto1979tight} \label{unique4des}Up to complementation, the $4$-$(23, 7, 1)$ design is the only tight $4$-$(m,S,\Lambda)$ design with $2<S<m-2$.
\end{theorem}
By Theorem \ref{thm3}, if $\mathcal D$ is  quasi-symmetric with $|B|={m\choose 2}$ ($4\leq S\leq m-4$), then it is the $4$-$(23, 7, 1)$ design by Theorem \ref{unique4des}. This fact will be used in the proof of the Main Theorem.  

\subsection{Coherent configurations}
Here, we review the notion of coherent configuration given by Higman in \cite{higman1987coherent}. The concept of the coherent configuration is  a purely combinatorial axiomatization of  permutation groups.
\begin{definition}Let $X$ be a finite set and let $\{R_i\}_{i\in I}$ be a set of relations on $X$ such that:
\begin{enumerate}
\item $\{R_i\}_{i\in I}$ is a partition of $X\times X$;

\item $R_i^t =R_{i^*}$ for some $i^*\in I$ where $R_i^t=\{(y,x): (x,y)\in R_i\}$;

\item there is a subset $\Omega\subset I$ such that $\{(x,x):x\in X\}=\bigcup_{i\in\Omega}R_i$;

\item given $(x,y)\in R_k$, $|\{z: (x,z)\in R_i, (z,y)\in R_j\}|$ is a constant $p_{ij}^k$ which depends only on $i,j,k$.
\end{enumerate}
Then $(X,\{R_i\}_{i\in I})$ is said to be a \emph{coherent configuration}.
\end{definition}
By the third condition, $X$ is split into several parts : $X_i=\{x:(x,x)\in R_i\}$ for $i\in\Omega$ and each part is called a \emph{fiber}. 
Let $s_{i,j}$ denote  the number of relations  in $X_i\times X_j$. The \emph{type} of coherent configuration is the matrix with size $|\Omega|\times |\Omega|$ and the $(i,j)$-entry is $s_{i,j}$.
When $|\Omega|=1$, it is  an \emph{association scheme}.
Each fiber carries a structure of association scheme.
\begin{example} Let $G=(V,E)$ be a regular graph that is neither complete nor empty. Then $G$ is said to be \emph{strongly regular} with \emph{parameters} $(n,k,\lambda,\mu)$ if 
it is of order $n$, $k$-regular, every pair of adjacent vertices has $\lambda $ common neighbors, and every pair of distinct nonadjacent vertices has $\mu$ common neighbors.
The concept of strongly regular graph and symmetric association scheme with three relations are equivalent. 
The eigenvalues of $G$ are listed from large to small by $k>r>s$. The eigenvalue $s$ must be a negative number. 
The reader can find details for strongly regular graphs in \cite{godsil2001algebraic}. 
\end{example}

\begin{example} Let $\mathcal D$ be a quasi-symmetric design with intersection number $\alpha,\beta$. Define $X=V\cup B$ and
\begin{align*}
    R_1&=\{(x,x):x\in V\};\\
    R_2&=\{(x,x):x\in B\};\\
    R_3&=\{(x,y):x,y\in V,x\neq y\};\\
    R_4&=\{(x,y)\in B\times B: |x\cap y|=\alpha\};\\
    R_5&=\{(x,y)\in B\times B: |x\cap y|=\beta\};\\
    R_6&=\{(x,y)\in V\times B: x \in y\};\\
    R_7&=\{(x,y)\in V\times B: x \not\in y\};\\
    R_8&=R_6^t;\,\,\,\,\,R_9=R_7^t.
\end{align*}
Then $(X, \{R_1,\dots,R_9\})$ is a coherent configuration of type $\begin{pmatrix} 2&2\\2 &3\end{pmatrix}$ (abbreviated as  (2,2;3)). The fibers are $V$ and $B$.
\end{example}

\subsection{Connection between type (2,2;3) and quasi-symmetric designs}
Higman showed in \cite{higman1987coherent} that a coherent configuration $(X_1\cup X_2,\{R_1,\dots,R_9\})$  of type (2,2;3) is equivalent to a complementary pair of quasi-symmetric design. The relation $R_6$ restricted in $X_1\times X_2$ carries a structure of quasi-symmetric design $\mathcal D$.   There is a relation $R$ in $\{R_4,R_5\}$ restricted in $X_2\times X_2$
equivalent to $\{(x,y)\in B\times B: |x\cap y|=\alpha\}$ where $\alpha$ is the largest intersection number of $\mathcal D$. Without loss of generality, we assume $R=R_4$. The $R_4$ restricted in $X_2\times X_2$
gives a structure of  strongly regular graph (which is called the \emph{block graph} of  $\mathcal D$).

\begin{proposition}\cite[section 9C]{higman1987coherent}\label{9C}  Let $\mathcal D$ be a $2$-$(m,S,\Lambda)$ quasi-symmetric design with intersection numbers $\beta<\alpha$. We can express for the parameters  of the block graph in terms of the parameters  $S,m,\alpha,\beta$ of $\mathcal D$. 
\begin{align*}
\Lambda=&\frac{S(S-1)(S-\alpha)(S-\beta)} {S^4-2S^3-((\alpha+\beta-1)(m-1)-1)S^2+\alpha\beta m(m-1)};\\
T	=&\frac{(m-1)\Lambda}{S-1};\,\,\,\,\,N	=\frac{\alpha(m-S)(S(S-1)-\beta(m-1))\Lambda }{S(\alpha-\beta)(S-\alpha)(S-1)}; 	\\
P	=& \frac{ (S(S-1)-\beta(m-1))  \Lambda}{(\alpha-\beta)(S-1)};\,\,\,\,\,r	=\frac{  1}{\alpha-\beta}\left( \frac{(m-S)\Lambda}{S-1}-(S-\beta) \right);\\
k	=&\frac{ (m-S)(S(S-1)-\beta(m-1)  )\Lambda }{  (\alpha-\beta) (S-\alpha)(S-1)  };\,\,\,\,\,n	=\frac{mT}S ;\,\,\,\, \,s	=\frac{\beta-S}{\alpha-\beta} .
\end{align*}

\end{proposition}

\begin{remark}
The parameters $S,m,\alpha,\beta$ are integers with 
\begin{align*}
        0\leq \beta<\alpha< S<m.
\end{align*} 
By the definition of $2$-design, $\Lambda\in \mathbb Z_{\geq 1}$. By the definition of $N,P$, they belong to $\mathbb Z_{\geq 0}$. 

\end{remark}

\subsection{Representation of coherent configurations}
The \emph{adjacency matrices} of a coherent configuration are the $|X|\times |X|$ matrices $A_i$ whose $(x,y)$-entry is $1$ if $(x,y)\in R_i$ and $0$ otherwise. The \emph{adjacency algebra}  is $\mathcal A=\text{span}_{\mathbb C}\{A_i:i\in I\}$.
Let $\triangle_1,\dots,\triangle_p$ be the non-isomorphic irreducible representations of degree, $e_1,\dots,e_p$ correspondingly.
There is a basis $\{\epsilon_{ij}^s\}$ of $\mathcal A$ defined by $\triangle_t(\epsilon_{ij}^s)=\delta_{st}E_{ij}^s$, where $E_{ij}^s$ is the $e_s\times e_s$ matrix with $(i,j)$-entry 1 and all other entries 0. They satisfy the equation $\epsilon_{ij}^s\epsilon_{kl}^t=\delta_{st}\delta_{jk}\epsilon_{il}^s$.

Higman showed in \cite{higman1987coherent} that  for a  type (2,2;3) coherent configuration, the adjacency algebra $\mathcal A$ has three  non-isomorphic  irreducible representations. For $x=c_1A_1+\cdots+c_9A_9\in \mathcal A$, the three representations   $\triangle _1,\triangle _2:\mathcal A\to \text{Mat}_2(\mathbb C),\triangle_3:\mathcal A\to \mathbb C$ are defined by  \begin{align*}
\triangle_1(x)&=\begin{pmatrix}
c_1+(m-1)c_3	& \alpha_1 c_6+\alpha_2 c_7\\
\alpha_1 c_8+\alpha_2c_9& c_2+kc_4+(n-k-1)c_5
\end{pmatrix},\\
\triangle_2(x)&=\begin{pmatrix}
c_1-c_3	& \beta_1c_6+\beta_2c_7\\
\beta_1c_8+\beta_2c_9	& c_2+rc_4-(r+1)c_5
\end{pmatrix},\\
\triangle_3(x)&=(c_2+sc_4-(s+1)c_5),
\end{align*}
 where  $\alpha_1=\sqrt{ST}, \alpha_2= P^{-1}(k-N)\sqrt{ST}, \beta_1= -\beta_2=\sqrt{T-\Lambda}$.

Let $\epsilon_{11}^2=c_1A_1+\cdots+c_9A_9$. Since $\triangle_s(\epsilon_{11}^2)=\delta_{s,2}E^2_{11}$, we have $(c_2+sc_4-(s+1)c_5)=0$,
\begin{align*}
\begin{pmatrix}
c_1+(m-1)c_3	& \alpha_1 c_6+\alpha_2 c_7\\
\alpha_1 c_8+\alpha_2c_9& c_2+kc_4+(n-k-1)c_5
\end{pmatrix}=\begin{pmatrix} 0 & 0\\0&0\end{pmatrix},\\
\begin{pmatrix}
c_1-c_3	& \beta_1c_6+\beta_2c_7\\
\beta_1c_8+\beta_2c_9	& c_2+rc_4-(r+1)c_5
\end{pmatrix}=\begin{pmatrix} 1&0\\0&0\end{pmatrix}.
\end{align*}
Now, we have $9$ variables and $9$ conditions, so we can solve $c_1,\dots,c_9$. Others $\epsilon_{ij}^s$ also have $9$ variables and $9$ conditions, so we can explicitly determine them. We list the solution of $\epsilon_{ij}^2$ below,
\begin{align}\label{epsilonbasis}
&\epsilon_{11}^2=	\frac{(m-1)A_1-A_3}m,\,\,\,\,\,	\epsilon_{12}^2=	\frac{ \alpha_2A_6-\alpha_1A_7}{\alpha_1\beta_2-\alpha_2\beta_1},\,\,\,\,\,\epsilon_{21}^2	=\frac{\alpha_2A_8-\alpha_1A_9}{\alpha_1\beta_2-\alpha_2\beta_1},\notag\\
&\epsilon_{22}^2=	\frac{-((n-k-1)s+ks+k)}{n(r-s)}A_2+\frac{n-k+s}{n(r-s)}A_4+\frac{s-k}{n(r-s)}A_5, 	
\end{align}
where $n,k,r,s$ are the parameters of the strongly regular graph given by $A_4|_{X_2\times X_2}$.

\subsection{Proof of Theorem \ref{embcond}}
\begin{proof}[Proof of Theorem \ref{embcond}:]
Let $  E:=(\epsilon_{11}^2+\epsilon_{12}^2+\epsilon_{21}^2+\epsilon_{22}^2)/2$ where $\epsilon_{ij}^2$ are given in (\ref{epsilonbasis}). Then $  E^2=  E$ and $  E^t=  E$.
The trace of $  E$ is $m-1$. Namely, $E$ is a symmetric idempotent of rank $m-1$. Let $\{e_x:x\in V\cup B\}$ be the standard basis of $\mathbb R^{V\cup B}$. Observe that $\langle  E e_x,  Ee_y\rangle= e_x^t  E^t  E e_y=e_x^t  Ee_y$ is the $(x,y)$-entry of $  E$. In other words,
\begin{align}\label{emb}
\langle  E e_x,  Ee_y\rangle=\begin{cases}
(m-1)/ (2m), &x,y\in V, x=y,\\
-1/(2m), &x,y \in V,x\neq y,\\
\alpha_2/(2\alpha_1\beta_2-2\alpha_2\beta_1),&x\in V,y\in B, x\in y,\\
-\alpha_1/(2\alpha_1\beta_2-2\alpha_2\beta_1),& x\in V,y\in B,x\not\in y,\\
-((n-k-1)s+ks+k)/(2n(r-s)),	&x,y\in B,x=y,\\
(n-k+s)/(2n(r-s)),	&x,y\in B, x\cap y=a,\\
(s-k)/(2n(r-s)), 	 &x,y\in B,x\cap y=b.
\end{cases}
\end{align}
Define $i:V\cup B\to\mathbb R^{V\cup B}$ by $i(x)=Ee_x$. 
Note that $i(V)$ is on one sphere and $i(B)$ is  on another  sphere. So, $|A(i(V),i(B))|$ is the number of different inner product $\langle i(v),i(b)\rangle$ for $v\in V$ and $b\in B$. From equation (\ref{emb}) we know that it is $2$. And the distance between $i(v)$ and $i(b)$ only depends on $v$ in $b$ or not. 
Similarly, $i(V)$ is an one-distance set on a sphere which is a regular simplex in $\mathbb R^{m-1}$.
Finally, the $F=\epsilon_{22}^2$ is an idempotent of rank $m-1$ on the algebra $\text{span}_{\mathbb C}\{A_2,A_4,A_5\}$. The inner product between points in $\mathbb R^B$ is 
\begin{align*}
    \langle i(x),i(y) \rangle=\frac1 2 e_x^t 2Ee_y=\frac 1 2e_x^tFe_y=\frac 1 2 e_x^t F^tFe_y=\frac 1 2 \langle Fe_x,Fe_y\rangle.
\end{align*}
So, the $i(B)$ is the scaling of the spherical embedding of a strongly regular graph.
Now, $i$ is an embedding satisfying the conditions in the statement of theorem. \end{proof}

\begin{remark}
In general,  this embedding has at most 5 distinct distances. If we fix the regular simplex and rescale the sphere that containing $i(B)$, it is still an embedding satisfying conditions of Theorem \ref{embcond}.
We will need smartly choosing the radius $R_2$ to make the embedding as a two-distance set. How to determine the $R_2$ will be discussed in the next section.
\end{remark}

\begin{example}
For the Lison\v ek's example, the parameters of $\mathcal D$ are $(S,m,\alpha,\beta,\Lambda,T,N,P)=( 2,9, 1, 0, 1, 8, 7, 2)$ and the parameters of the block graph are $(n,k,\lambda,\mu)=(36,14,7,4)$. Now, 
\begin{align}
\langle  E e_x,  Ee_y\rangle=\begin{cases}
4/9, &x,y\in V, x=y,\\
-1/18, &x,y \in V,x\neq y,\\
-\sqrt 7 /18,&x\in V,y\in B, x\in y,\\
\sqrt 7/63,& x\in V,y\in B,x\not\in y,\\
1/9,	&x,y\in B,x=y,\\
5/126,	&x,y\in B, x\cap y=a,\\
-2/63, 	 &x,y\in B,x\cap y=b.
\end{cases}
\end{align}
Hence, the radius $R_1$ of the sphere containing $EV$ is $2/3$ and the radius $R_2$ of the sphere containing $EB$ is $1/3$.
$A(EV,EV)=\{1\}$, $A(EV,EB)=\left\{ \frac{\sqrt{\sqrt 7+5}}3, \sqrt{\frac 5 9-\frac{2\sqrt7}{63}} \right\}$, $A(EB,EB)=\{\sqrt{1/7}, \sqrt{2/7}\}$.
After rescaling $R_1$ with $\sqrt 2$ and $R_2$ with $\sqrt{14}$, $E(V\cup B)$ becomes a two-distance set where $A(E(V\cup B),E(V\cup B))=\{\sqrt{2},2\}$.
\end{example}

\section{Embedding as a two-distance set} 
 
 Let $\mathcal D=(V,B)$ be a quasi-symmetric design  and $i$ be an embedding satisfying conditions in Theorem \ref{embcond}. In this section, we assume that $|A(  i(X),  i(X) )|=|\{\sqrt{2},\sqrt{\gamma}\}|=2$ where $X=V\cup B$. The following theorem is the main theorem of this section.

\begin{theorem}\label{thm1}
Let $(V\cup B, \{R_1,\dots,R_9\})$ be a coherent configuration of type (2,2;3).
Let $\mathcal D=(V,B)$ be the corresponding  $2$-$(m,S,\Lambda)$ quasi-symmetric design  with two intersection numbers $x$ and $y$.
Suppose  the  embedding is  a two-distance set and satisfies the conditions of Theorem \ref{embcond}.
Then   $p_1(S,m,x,y)=p_2(S,m,x,y)=p_3(S,m,x,y)=0$,  where 
\begin{align*}
p_1(S,m,x,y)=& S^4 - 2S^2xm + x^2m^2 - 2S^3 + 2S^2x - 2Sxm + 2x^2m + S^2 - 2Sx + x^2;\\
p_2(S,m,x,y)= &S^4 + 2S^2xm - 4S^2ym + x^2m^2 - 4xym^2 + 4y^2m^2 - 2S^3 + 2S^2x \\
&+ 8S^2m - 6Sxm - 2x^2m- 4Sym + 4xym - 4Sm^2 + 4xm^2 + S^2 - 2Sx + x^2;\\
p_3(S,m,x,y)=&S^2x^2 - Smx^2 - 2S^2xy + 2Smxy + S^2y^2 - Smy^2 + S^2m \\
&+ 2S^2x - 2Smx - 2Sx^2 + mx^2 - 2S^2y + 2Sxy + S^2 - 2Sx + x^2.
\end{align*}
\end{theorem}

\subsection{Calculation of A(i(X),i(X))}
Let $A( i( V), i( V))=\{d_1\}$, $A( i( B),  i(B))=\{d_2,d_3\} $ and $A( i( V), i( B))=\{d_4,d_5\}$ with $d_2<d_3$ and $d_4<d_5$.

Let $ i( V)=\{e_1,\dots,e_m\}$ be the standard orthonormal basis. It forms a regular simplex of $\mathbb R^{m-1}$ located on the affine hyperplane 
$$ H=\left\{ x\in\mathbb R^{m}: \sum_{i=1}^{m}x_i=1\right\}\subset \mathbb R^m.$$
For convenience, we  define $f(x)=(x-1)/x$ for all $x\neq 0$. We will determine the elements in $A(i(X),i(X))$ expressed by the parameters $S,m,\alpha$, and $\beta$.

\begin{theorem}\label{thm:A(X,X)} The points set $i(V)$ and $i(B)$ are located on the spheres with radius $R_1$
and $R_2$ respectively. Then   
\begin{align*}
R_1&=\sqrt{f(m)},\\
R_2&=\begin{cases}
\sqrt{2-f(m-S)}+\sqrt{f(m)-f(m-S)}, &\gamma>2,\\
\sqrt{2-f(m-S)}-\sqrt{f(m)-f(m-S)},	&\gamma<2,
\end{cases}\\
   d_1&=\sqrt{2},\,\,\,\,\,d_2=R_2\sqrt{\frac{2(S - \alpha)m}{(m-S)S}},\,\,\,\,\, d_3=R_2\sqrt{\frac{2(S - \beta)m}{(m-S)S }},\\
   d_4&=\begin{cases}
   \sqrt{ R_2^2-2R_2\sqrt{ f(m)-f(m-S) }  +f(m)}, &\text{ if }\gamma>2,\\
   \sqrt{ R_2^2-2R_2\sqrt{ f(m)-f(S) }   +f(m)},&\text{ if }\gamma<2
   \end{cases}\\
d_5&=\begin{cases}
\sqrt {R_2^2+2R_2\sqrt{   f(m)-f(S) }+f(m)},&\text{ if }\gamma>2,\\
\sqrt {R_2^2+2R_2\sqrt{   f(m)-f(m-S) }+f(m)},&\text{ if }\gamma<2.
\end{cases}
\end{align*}
\end{theorem}

The Theorem \ref{thm:A(X,X)} is the consequence of Lemma \ref{par1} and Lemma \ref{par2}. 

\begin{lemma}\label{par1} The radius $R_1$ and the distances $d_1,d_2,$ and $d_3$ between the points  satisfy the formula in Theorem \ref{thm:A(X,X)}.
\end{lemma}

\begin{proof}
We assume that the vertices of the regular simplex are  $e_1,e_2,\dots,e_m$ and the center of the simplex is $(\frac 1 m, \frac 1 m, \cdots, \frac 1 m)$. With easy calculation, we can obtain $R_1=\sqrt{\frac{m-1}{m}}=\sqrt{f(m)}$ and $d_1=\sqrt 2$. 
We require the $ i( B)$ as the spherical embedding of a strongly regular graph. If it is on the sphere with radius $R_2$, then the elements in $A(i(B),i(B))$ are
\begin{align*}
 R_2\sqrt{2-2 \frac r k},\,\,\,\,\, R_2\sqrt{2-2\frac{-1-r} {n-k-1}}
\end{align*}
due to the formula (\ref{emb}) (the block graph is neither complete nor empty,  $n-k-1>0$ and $k>0$). If we substitute $r, k$ by $S,m,\alpha,\beta$, assuming $\beta<\alpha$ and $d_2<d_3$,
we have $d_2= R_2\sqrt{\frac{2(S - \alpha)m}{(m-S)S}} $ and $d_3=R_2\sqrt{\frac{2(S - \beta)m}{(m-S)S }} $.
\end{proof}

The centroid of the regular simplex $i(V)$  is $ o=\frac 1{m} (1,\dots,1)$. Take $v,w\in V$ and $b\in B$ with $v\in b$ and $w\not\in b$. Let $$p=\frac{1}{S}\sum_{\substack{t\in V\\ t\in b}}   i(t),\,\,\,\,\,  q=\frac 1 {m-S}\sum_{\substack{t\in V\\ t\not\in b}}  i(t)$$ be the centroids of points in $b$ and points not in $b$. Direct computation shows that $\langle  i( v)- p,p- o\rangle=0=\langle o- q,  q- i( w)\rangle$. Geometrically, we may regard that the centroid of $v$ will lie in the perpendicular bisector plane.  We have
$ \|  i(v)-p\|_2= \| (1,0, \cdots, 0) - (\frac 1 S, \frac 1 S, \dots, \frac 1 S, 0, \dots, 0)\|_2= \sqrt {\frac{S-1}{S}}= \sqrt{f(S)} $. Similarly, we can obtain $\|  i(w)-q\|_2=\sqrt{f(m-S)}
$. To visualize our calculation, we give the figure \ref{fig4} in the following.
\begin{figure}[htb]
  \centering
\includegraphics[scale=0.6]{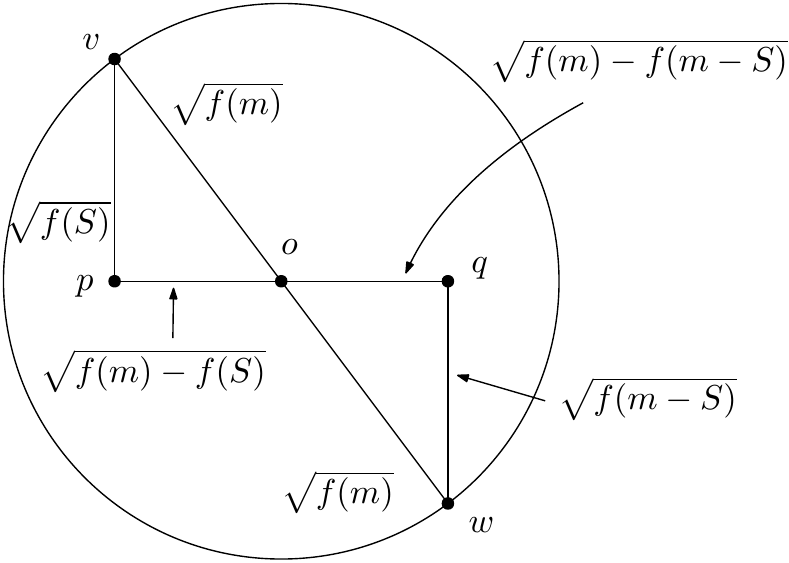}
    \caption{length of $\|i(v)-p\|_2,\|p-o\|_2,\|o-q\|_2,\|q-i(w)\|_2$}\label{fig4}
\end{figure}

We divide $i$   into $4$ cases. If $i$ belongs to case $t$, then use  $\iota_t$ denote $i$.
First, the embedding of vertices located in a smaller sphere or the other way. Let $\iota_1,\iota_2$ send vertices to the smaller sphere and $\iota_3,\iota_4$ send vertices to the larger sphere.  Second, $\iota_1,\iota_4$ send elements not in block located closer than elements in block, and $\iota_2,\iota_3$ sends elements in block located closer than elements not in block.
We show the $4$ categories geometrically in the Figure \ref{fig1}.
\begin{figure}[htb]
    \begin{minipage}[t]{.45\textwidth}
        \centering
\includegraphics[scale=0.6]{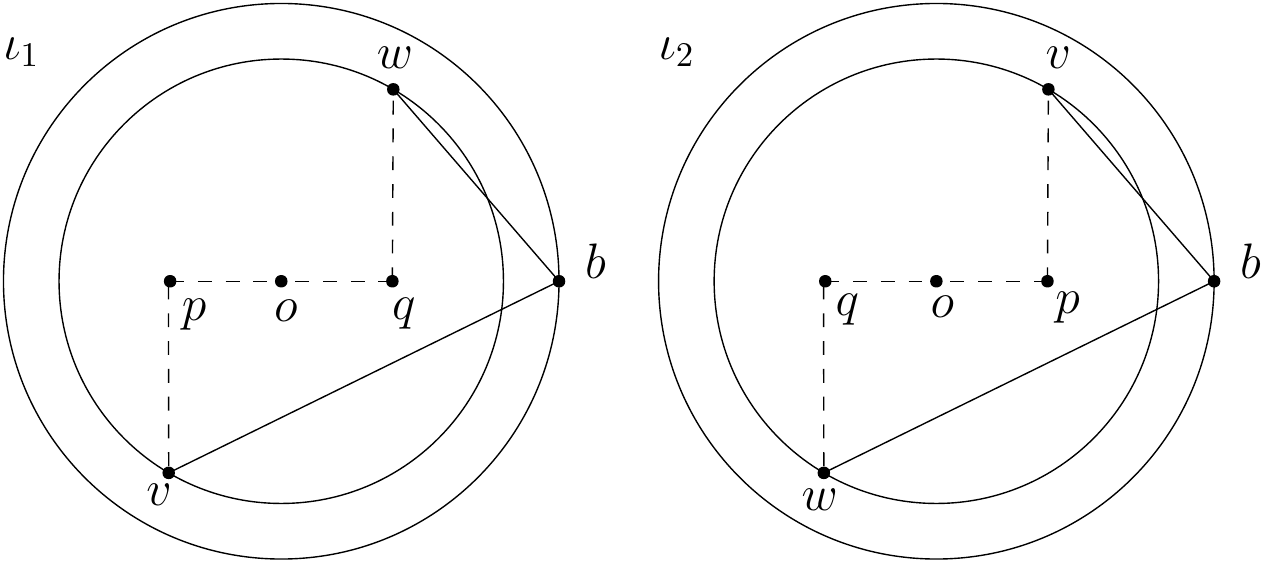}
    \end{minipage}
    \hfill
    \begin{minipage}[t]{.45\textwidth}
        \centering
\includegraphics[scale=0.6]{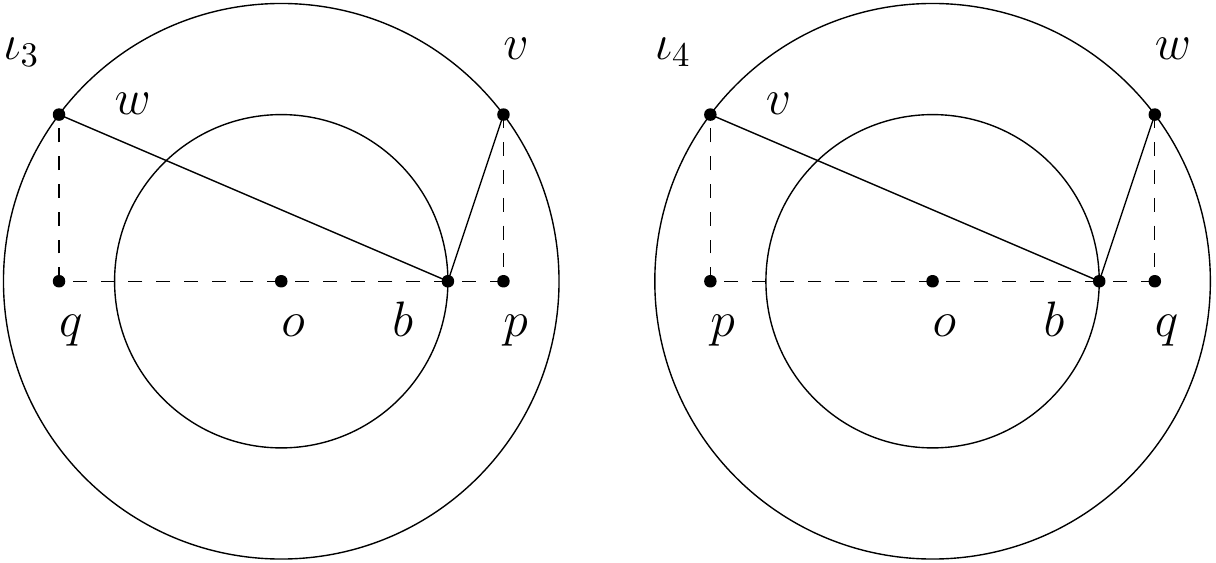}
    \end{minipage}  
    \caption{$\iota_1,\iota_2,\iota_3,\iota_4$}\label{fig1}
\end{figure}
\begin{figure}[htb]
        \centering
\includegraphics[scale=0.5]{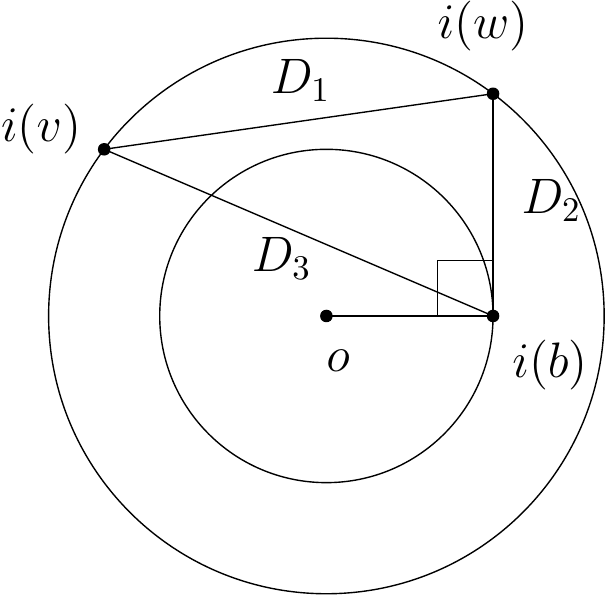}
    \caption{$i(b)=q$ }
    \label{fig2}
\end{figure}

\begin{remark} It is impossible that $i(b)=q$ (see the Figure \ref{fig2}). We will explain it in the following. 
Let $D_1=\|i(v)-i(w)\|_2$, $D_2=\|i(w)-i(b)\|_2$ and $D_3=\|i(v)-i(b)\|_2$.
We already knew $D_1=\sqrt 2$ and $D_2=\sqrt{f(m-S)}$. Now,
\begin{align*}
    D_3=\left\|\left(1,\underbrace{0,\dots,0}_{S-1 \text{ terms}},\underbrace{0,\dots, 0}_{m-S\text{ terms}}\right)-\frac 1{m-S}\left(\underbrace{0,\dots,0}_{S \text{ terms}},\underbrace{1,\dots, 1}_{m-S\text{ terms}}\right)\right\|_2=\sqrt{\frac{m-S+1}{m-S}}>1.
\end{align*}
Hence, $D_3>D_2$. According to the  two-distance set assumption, $D_3=D_1=\sqrt{2}$  implies $m-S=1$ and $D_2=\sqrt{f(m-S)}=0$. Hence, we send every point not in block to the same place which is a contradiction.

\end{remark}

\begin{lemma}\label{par2} The radius $R_2$ and the distances  $d_4$ and $d_5$ between the points satisfy the formula in Theorem \ref{thm:A(X,X)}.
\end{lemma}

\begin{proof}

According to $\gamma>2$ or $\gamma<2$,  there are total $8$ cases need to be considered.
We listed them in Table \ref{tab2}.

\begin{table}[h]
\centering
\begin{tabular}{ccccc}
	&$\iota_1$&$\iota_2$&$\iota_3$&$\iota_4$\\
$\gamma>2$ &(A)&(B)&(C) &(D) \\
$\gamma<2$ &(E)&(F)&(G)&(H)
\end{tabular}
\caption{$8$ cases}\label{tab2}
\end{table}
Now, we can determine the $d_4$ and $d_5$. 
Take $v,w\in V$ and $b\in B$ with $v\in b$ and $w\not\in b$. According to Figure \ref{fig4}, for case 1 and case 4, we have
$d_5^2=\|i(v)-i(b)\|^2_2=\|p-i(v)\|_2^2+ (R_2+\|o-p\|_2)^2 =f(S)+(R^2_2+\sqrt{f(m)-f(S)} )^2
$. Also, $d_4^2=\|i(w)-i(b)\|^2_2=\|q-i(w)\|_2^2+ (R_2-\|o-q\|_2)^2 =f(m-S)+(R^2_2+\sqrt{f(m)-f(m-S)} )^2$. 
Hence,
\begin{align*}
&\|i(b)-i(v)\|_2 = \begin{cases} 
d_4=\sqrt{ R_2^2-2R_2\sqrt{ f(m)-f(m-S) }   +f(m)}, &\text{ if }v\not \in b;\\
d_5=\sqrt {R_2^2+2R_2\sqrt{   f(m)-f(S) }+f(m)},&\text{ if } v\in b.
\end{cases}
\end{align*}
For case 2 and case 3, 
\begin{align*}
&\|i(b)-i(v)\|_2 = \begin{cases} 
d_4=\sqrt{ R_2^2-2R_2\sqrt{ f(m)-f(S) }   +f(m)}, &\text{ if }v\in b;\\
d_5=\sqrt {R_2^2+2R_2\sqrt{   f(m)-f(m-S) }+f(m)},&\text{ if } v\not\in b.
\end{cases}
\end{align*}
Denote the complement of $\mathcal D$ as $\overline{\mathcal D}$. 
If the embedding of $\mathcal D$ is case 1, the embedding of $\overline{\mathcal D}$ is case 2. Moreover, $\iota_1(\mathcal D)=\iota_2(\overline{\mathcal D})$.
The same is true for  case 3 and case 4. 
This reduces to four cases. Without loss of generality, we consider cases (A), (D), (F), (G).

Suppose $\gamma>2$.
From the formula for $d_4, d_5$, we find that $R_2$  is a root of the degree two  polynomial 
$x^2+2x\sqrt{f(m)-f(m-S)} +f(m)-d_4^2$. But, since $f(m)<2$, the only positive root is
$ \sqrt{ f(m)-f(m-S)} + \sqrt{2- f(m-S) } $.
Since this number greater than $R_1$ we conclude that case (D) is impossible.

Suppose $\gamma<2$.
From the formula for $d_4, d_5$, we find that $R_2$ is a root of the polynomial 
$x^2-2x\sqrt{f(m)-f(S)} +f(m)-d_4^2$. The only positive root is
$ \sqrt{2-f(m-S)}-\sqrt{f(m)-f(m-S)} $.
Since this number is less than $R_1$, we conclude that case (F) is impossible. \end{proof}

\begin{remark}
The parameters of conference graphs are
$(n,(n-1)/2,(n-5)/4,(n-1)/4)$. If the block graph is a conference graph, then $d:=m-1$ is the multiplicity of $r$. Hence, $n=2m-1$. Now, the $|X|=m+n>2d+3$ (when $m>2$). By \cite[Theorem 2]{larman1977two} and Theorem \ref{thm:A(X,X)}, $s\in\mathbb Z$ (and hence $r\in\mathbb Z$).
\end{remark}

\subsection{Proof of Theorem \ref{thm1}}
\begin{proof}[Proof of Theorem \ref{thm1}] Let $A(i(X),i(X))=\{\sqrt 2,\sqrt{\gamma}\}$ and let $(S,m,\alpha,\beta)$ be the parameters of $\mathcal D$. 
 We will prove the following statement:
For $\gamma>2$, 
$p_1(S,m,\alpha,\beta)=p_2(S,m,\alpha,\beta)=p_3(S,m,\alpha,\beta)=0$ and for $\gamma<2$, 
$p_1(S,m,\beta,\alpha)=p_2(S,m,\beta,\alpha)=p_3(S,m,\beta,\alpha)=0$.

If $\gamma>2$, then the embedding is of case 1. Since $2=d_2^2=R_2^2 ( 2(S-\alpha)m)/(mS-S^2) $, we have $R_2=\sqrt{\frac{ (m-S)S}{(S-\alpha)m} } $ and 
$
d_4^2 =2=R_2^2-2R_2\sqrt{f(m)-f(m-S)} +f(m)$ and equivalent to 
$ 2S \sqrt{S-\alpha} = - (S^2 - \alpha m + S - \alpha)$ which implies
$ p_1(S,m,\alpha,\beta)=0$.

Next,
$
d_3^2=d_5^2$ if and only if $ R_2^2   (2(S -\beta )m  )/(mS -S^2 ) = R_2^2+2R_2\sqrt{f(m )-f(S )} +f(m ) $ and equivalent to  $2(m -S ) \sqrt{S -\alpha } =-(S ^2 + \alpha m  - 2\beta m  + S  - \alpha )$ which implies $p_2(S ,m ,\alpha ,\beta )=0$.
By Lemma \ref{par1}, $R_2=\sqrt{2-f(m -S )}+\sqrt{f(m )-f(m -S )}$. 
By Theorem \ref{thm:A(X,X)},
$
\frac{d_2^2}{d_3^2}=\frac{-s-1}{-s}.
$ So,
$
\frac{d_2^2}{d_3^2}=f(-s )$ if and only if $$ \frac{2}{ R_2^2+2R_2\sqrt{f(m )-f(S )}+f(m )} =\frac{S -\alpha }{S -\beta }$$ which 
implies that 
$ p_3(S ,m ,\alpha ,\beta )=0$.

If $\gamma<2$, then embedding is of case 3.
We apply a similar argument. Now, $R_2=\sqrt{\frac{(m-S)S}{(S-\beta)m}}$. Hence, $d_3^2=d_5^2$ implies $p_1(S,m,\beta,\alpha)=0$ and $d_2^2=d_4^2$ implies $p_2(S,m,\beta,\alpha)=0$.
Finally, $R_2=\sqrt{2-f(m-S)}-\sqrt{f(m)-f(m-S)}$ and 
$
\frac{d_2^2}{d_3^2}=f(-s) =\frac{S-\alpha}{S-\beta}$ implies $ p_3(S,m,\beta,\alpha)=0
$.

\end{proof}

\begin{remark}\label{rem4.2}
If $p_1(S ,m ,0,y )=0$, then $S \in\{0,1\}$. Now, $p_2(1,m ,0,y )=4(my + m - 2)m(y - 1)$ implies $m \in\{0,2/(y +1)\}$ or $y =1$. Finally, $p_3(1,\frac 2 {y +1},0,y )=(y -1)(y -3)$. Since $0\leq \beta <\alpha <S <m $, we can't find any feasible solution. 
\end{remark}

\section{Integer solutions of the polynomials system}

We want to find integer solutions of 
$ p_i(S_0,m_0,x_0,y_0)=0$ for $i=1,2,3$, in conjunction of that all the parameters $ S_0,m_0,y_0,\Lambda_0,T_0,n_0,k_0,r_0,s_0,l_0,N_0,P_0$ are integers, where
\begin{align*}
p_1(S,m,x,y)=& S^4 - 2S^2xm + x^2m^2 - 2S^3 + 2S^2x - 2Sxm + 2x^2m + S^2 - 2Sx + x^2;\\
p_2(S,m,x,y)= &S^4 + 2S^2xm - 4S^2ym + x^2m^2 - 4xym^2 + 4y^2m^2 - 2S^3 + 2S^2x \\
&+ 8S^2m - 6Sxm - 2x^2m- 4Sym + 4xym - 4Sm^2 + 4xm^2 + S^2 - 2Sx + x^2;\\
p_3(S,m,x,y)=&S^2x^2 - Smx^2 - 2S^2xy + 2Smxy + S^2y^2 - Smy^2 + S^2m \\
&+ 2S^2x - 2Smx - 2Sx^2 + mx^2 - 2S^2y + 2Sxy + S^2 - 2Sx + x^2.
\end{align*}

It would not be an easy task, but we can completely determine the solutions in the following steps. First, by the condition $p_1=p_2=p_3=0$, we could express all parameters $\Lambda, T, n, k, r, s,  N, P$ given in Proposition \ref{9C} by the variables $x$ and $z$. Then, we will define a magic auxiliary function $g(x,z)$ such that the function $g(x,z)$ could be bounded values on the unbounded domain of the $zx$ plane. Since $g$ is an integer and bounded, then there are only finitely many cases to be discussed and analysis. Therefore, we can completely determine the solution set. 

Notice that the subscript $|_0$ on these variables mean that the variables are evaluated at some particular values. For instance, $x_0$ is the notation for given value of variable $x$.
Let
\begin{equation} \label{eq:unprovide}
z := \frac{x (m + 1) - S (S + 1)}{2 S}.
\end{equation}

Lemma \ref{lem:semiportable} allows us to rewrite the parameters in Proposition \ref{9C} as rational functions in just two variables, $x$ and $z$. 
\begin{lemma} \label{lem:semiportable}
Assume that $S, m, x \neq 0$, $p_1(S, m, x, y) = 0$ and $p_2(S, m, x, y) = 0$. Then,
\begin{align}
\label{eq:baud}
S & = x + z^2, \\
\label{eq:bronchioli}
m & = (x + z^2 + z)^2 / x, \\
\label{eq:oligandrous}
y & = \text{$y_1$ or $y_2$},
\end{align}
where
\begin{align*}
y_1 & = x - z, \\
y_2 & = \frac{x^3+2 x^2 z^2+x^2 z+x z^4+2 x z^3+3 x z^2+z^5+2 z^4+z^3}{\left(x+z^2+z\right)^2}.
\end{align*}
Moreover, if $S_0, m_0, x_0, y_0$ are nonzero integers such that $p_1(S_0, m_0, x_0, y_0) = p_2(S_0, m_0, x_0, y_0) = 0$, then $z_0 := z|_{S = S_0, m = m_0, x = x_0, y = y_0}$ is an integer.
\end{lemma}

\begin{proof}
From (\ref{eq:unprovide}), we get
\begin{equation} \label{eq:twiller}
m = \frac{S (S + 1 + 2 z)}{x} - 1.	
\end{equation}
Equation (\ref{eq:baud}) follows from substituting  (\ref{eq:twiller}) into $p_1 / S^2 = 0$. Substitute equation (\ref{eq:baud}) back into  (\ref{eq:twiller}) and we get  (\ref{eq:bronchioli}). Substitute both equation (\ref{eq:baud}) and (\ref{eq:bronchioli}) into $p_2 / m = 0$, and we get a quadratic equation in $y$, which has two solutions of the equation (\ref{eq:oligandrous}).

Now let $S_0, m_0, x_0, y_0$ be integers satisfying the assumptions. We see from (\ref{eq:baud}) that $z_0^2 = S_0 - x_0$ is an integer, and from (\ref{eq:baud}) and (\ref{eq:bronchioli}) that $(S_0 + z_0)^2 = m_0 x_0$ is an integer. Therefore, $z_0$ is an integer.
\end{proof}

There are two cases that $y = y_1$ and $y = y_2$.

For $y=y_1$ we discuss in Proposition \ref{prop:did} and $y=y_2$ in Proposition \ref{prop:retinker}.

\begin{proposition} \label{prop:did}
Treat all parameters $\Lambda, T, n, k, r, s,  N, P$ given in Proposition \ref{9C} as elements in $\R(x, z)$ by Lemma \ref{lem:semiportable} and $y = y_1$. Let $x_0$ be a positive integer and $z_0$ an integer such that $\Lambda_0, T_0, n_0, k_0, r_0, s_0, l_0, N_0, P_0$ are all integers. Then, either $x_0 = z_0 (z_0 + 1) / 2$, or $z_0 = 0$, or $z_0 = -1$.
\end{proposition}

\begin{proof}
Consider the auxiliary polynomial
\[
	g := 3 + x + 19 z + 16 z^2 + 3 k + 3 \Lambda - m - 4 n - 18 P + 21 z r.
\]
Since this is a polynomial of parameters with integral coefficients, $g_0$ is an integer.
Observe that $$\mathbb R^2=\bigcup_{z,k\in\mathbb R}(z,z(z+1)/2+k).$$ 
We want to find $(z,x)\in \mathbb Z\times\mathbb Z_{\geq 1}$, so $k\in\mathbb Z$.
Use computers to prove that:
\begin{itemize}
\item(Region1): $g \in (0, 1)$ when $z \in (-\infty, -2] \cup [1, +\infty)$, $x \geq 1$ and $x \geq z (z + 1) / 2 + 1$;
\item(Region2): $g \in (0, 2)$ when $x \geq 1$ and $x = z (z + 1) / 2 - 1$;
\item(Region3): $g \in (0, 1)$ when $x \geq 1$ and $x \leq z (z + 1) / 2 - 2$.
\end{itemize}
\begin{figure}[h]
    \centering
    \includegraphics[scale=0.7]{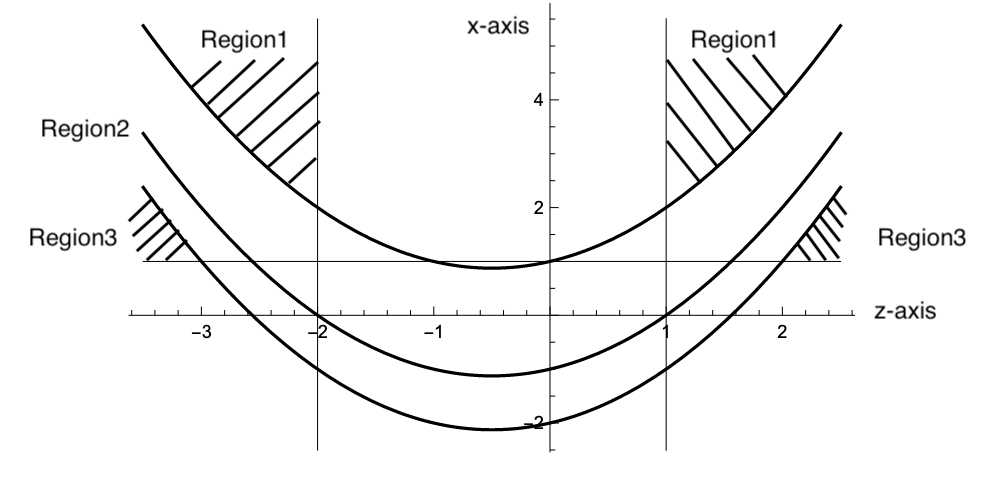}    
    \caption{Regions}
    \label{area}
\end{figure}
The regions are showed in Figure \ref{area}. By the facts above and that $g_0$ is an integer, if $(x_0, z_0)$ does not satisfy the conclusion of this lemma, then the only possibility is that $x_0 = z_0 (z_0 + 1) / 2 - 1$ and $g_0 = 1$. Solving this system of equations would give us $z_0 = \frac{1}{2}(- 1 \pm \sqrt{41})$, which contradicts the fact that $z_0$ is an integer.
\end{proof}

\begin{proposition} \label{prop:retinker}
Under the same assumption of Proposition \ref{prop:did} except that $y = y_2$, we assume in addition that $p_3(S_0, m_0, x_0, y_0) = 0$. Then, $z_0 = 0$.
\end{proposition}

\begin{proof}
Consider the auxiliary polynomial
\[
	g := -72 \Lambda + 13 m + 13 n + 99 T - 45 x + 32 y - 14 z - 13 m z + 39 n z - 13 T z + 13 x z - 33 z^2 + 13 z^3.
\]
Since it is an integer coefficient polynomial in parameters, $g_0$ is an integer.

Use computers to prove that:
\begin{itemize}
\item $g \in (31, 33)$ when $z \in (-\infty, -15] \cup [10, +\infty)$ and $x \in \{1, 2\}$.
\item $g \in (-1, 38)$ when $x \geq 3$.
\end{itemize}

\noindent Case 1. $x_0 \in \{1, 2\}$.

An enumeration of small pairs $(x_0, z_0)$ shows that there are no small integral pair which gives $g_0 = 32$.
For $x_0\in\{1,2\}$ and $z_0\in \{-14,\dots,9\}$, only $ g(1,-1)=136, g(1,0)=0 ,g(2,0)=0$ are integers. But, $\Lambda(1,-1)=-1$. 

\noindent Case 2. $x_0\geq3$ and $g_0 \in \{1, \dots, 37\}$.

For each $i \in \{1, \dots, 37\}$, the intersection of the curves $p_3 = 0$ and $g = i$ consists of finitely many points. An explicit calculation shows that the only integral points are $(x_0, z_0) = (0, 1)$ and $(x_0, z_0) = (-9, 3)$. Both violate the assumption that $x_0$ is positive.

\noindent Case 3. $x_0\geq 3$ and $g_0 = 0$.

The intersection of the curves $p_3 = 0$ and $g = 0$ consists of a curve $z = 0$, and $8$ points, all of which are not integral points. \end{proof}

\begin{theorem}\label{solufamily}
Consider the system of equations $p_1(S, m, x, y) = p_2(S, m, x, y) = p_3(S, m, x, y) = 0$. The only integral solutions $(S_0, m_0, x_0, y_0)$ such that $S_0, m_0 \neq 0$, $x_0 \geq 1$ and $\Lambda_0,$ $T_0,$ $n_0,$ $k_0,$ $r_0,$ $s_0,$ $l_0,$ $N_0,$ $P_0$ are all integers are on the following parametrized curves.
\begin{enumerate}
	\item[(i)] $S = \frac{1}{2} z (3 z + 1)$, $m = \frac{9}{2} z (z + 1)$, $x = \frac{1}{2} z (z + 1)$, $y = \frac{1}{2} z (z - 1)$.
	\item[(ii)] $S = z$, $m = z$, $x = z$, $y = z$.
	\item[(iii)] $S = z + 1$, $m = z$, $x = z$, $y = z + 1$.
\end{enumerate}

\end{theorem}
\begin{proof}
Recall that 
\begin{align*}
S & = x + z^2, \\
m & = (x + z^2 + z)^2 / x, \\
y & = \text{$y_1$ or $y_2$},
\end{align*}
where
\begin{align*}
y_1 & = x - z, \\
y_2 & = \frac{x^3+2 x^2 z^2+x^2 z+x z^4+2 x z^3+3 x z^2+z^5+2 z^4+z^3}{\left(x+z^2+z\right)^2}.
\end{align*}
For the case $y=y_1$,  we have $x_0 = z_0 (z_0 + 1) / 2$, or $z_0 = 0$, or $z_0 = -1$  by Proposition \ref{prop:did}.
If $x = z (z + 1) / 2$, then $S = \frac{1}{2} z (3 z + 1)$, $m = \frac{9}{2} z (z + 1)$, $x = \frac{1}{2} z (z + 1)$, $y = \frac{1}{2} z (z - 1)$.
If $z=0$, then $S=m=y=x$. If $z=-1$, then $S=y=x+1$ and $m=x$.

For the case $y=y_2$, we have $z_0 = 0$ by Proposition \ref{prop:retinker}.
Thus,   $S=m=y=x$ holds.\end{proof}

\section{Proof of main Theorem}

\begin{proof}[Proof of the Theorem \ref{main thm}] Let $\mathcal D$ be a $2$-$(m_0,S_0,\Lambda_0)$ quasi-symmetric design satisfying the conditions in Theorem \ref{thm1}.
It follows from Theorem \ref{thm1} that 
$$ p_1(S_0,m_0,x_0,y_0)=p_2(S_0,m_0,x_0,y_0)=p_3(S_0,m_0,x_0,y_0)=0.$$
By Remark \ref{rem4.2}, $x_0\geq 1$.
From Theorem \ref{solufamily}, we have three possible solutions need to be discussed.
Parametric (ii) and (iii) give $S_0=\alpha_0$ which is a contradiction.
It remains to consider the Parametrized solution (i).
From Parametrized solution (i):
\begin{align*}
S = \frac{1}{2} z (3 z + 1),\,\,\, m = \frac{9}{2} z (z + 1),\,\,\, x = \frac{1}{2} z (z + 1),\,\,\, y = \frac{1}{2} z (z - 1)
        \end{align*}
By Remark \ref{rem4.2} and Lemma \ref{lem:semiportable}, $z_0$ should be an integer.

Suppose $2\leq S_0 \leq 3$.
 If $S_0=2$, then  $z_0=1$ or $-4/3$. Lemma \ref{lem:semiportable} implies $z_0=1$ which is corresponding to the Lison\v ek's example.
 If $S_0=3$, then $z_0=\pm \sqrt{73}/6 - 1/6$ which is impossible.
 
 Suppose $4\leq S_0\leq m_0-4$. Since $n_0={m_0\choose 2}$, by Theorem \ref{thm3}, $\mathcal D$ is the $4$-$(23,7,1)$ design. Now, $S_0=7$ implies $z_0\in \{-7/3,2\}$.
If $z_0=2$, then $m_0=27\neq 23$. 
 
 Suppose $S_0=m_0-3$, then
$z_0=\pm\sqrt{13}/3 - 2/3$ which is impossible.

 Suppose $S_0=m_0-2$, then
 $z_0=\pm\sqrt{10}/3 - 2/3$ which is impossible.

 Suppose $S_0=m_0-1$, then
 $z_0=\pm \sqrt 7/3 - 2/3$ which is impossible.
\end{proof}

\section{Discussions}
We would like to propose some further research problems. 

(1) In Nozaki-Shinohara \cite{nozaki2020maximal}, they consider the two-distance
sets in $\mathbb R^d$ which contain a regular simplex and a strongly regular graph. We believe that
the case where:
the regular simplex is of size $d+1$ and the strongly regular
graph comes from a natural embedding (with respect to a
primitive idempotent of rank $d$) would be the most interesting (and extremal) case. We discussed this problem assuming the additional
condition that the two-distance set has the structure of a
coherent configuration of type (2,2;3). We wonder whether it is
possible to drop this additional condition on the existence of
the coherent configuration.

(2) It would be interesting whether there exists any
two-distance set in $\mathbb R^d$ coming from the natural embedding
of a coherent configuration of type (3,2;3).

(3) It is natural, although it is not so easy, to try to
generalize the discussion on two-distance sets to
three-distance sets in some way. For example, can we
classify three-distance set in $\mathbb R^d$ coming from the
natural embedding of a coherent configuration of
type (2,2;4). (There are many other possibilities.)

(4) Although it seems to be a difficult problem,
it would be interesting to study two-distance set $X$ in
$\mathbb R^d$ of the maximum cardinality ${d+ 2 \choose 2},$ whether
we can find a structure of a coherent configuration,
or some combinatorial structure close to the coherent
configuration?

\bibliographystyle{amsplain}
\bibliography{main}

\providecommand{\bysame}{\leavevmode\hbox to3em{\hrulefill}\thinspace}
\providecommand{\MR}{\relax\ifhmode\unskip\space\fi MR }
\providecommand{\MRhref}[2]{%
  \href{http://www.ams.org/mathscinet-getitem?mr=#1}{#2}
}
\providecommand{\href}[2]{#2}
\begin{thebibliography}{10}

\bibitem{bannai1983upper}
Eiichi Bannai, Etsuko Bannai, and Dennis Stanton, \emph{An upper bound for the
  cardinality of an s-distance subset in real euclidean space, ii},
  Combinatorica \textbf{3} (1983), no.~2, 147--152.

\bibitem{bannai2021classification}
Eiichi Bannai, Etsuko Bannai, Ziqing Xiang, Wei-Hsuan Yu, and Yan Zhu,
  \emph{Classification of spherical $2$-distance $\{4, 2, 1\}$-designs by
  solving diophantine equations}, Taiwanese Journal of Mathematics \textbf{25}
  (2021), no.~1, 1--22.

\bibitem{blokhuis1983few}
Aart Blokhuis, \emph{Few-distance sets}, PhD thesis, Eindhoven University of
  Technology (1983).

\bibitem{cameron1975graph}
Peter~Jephson Cameron and Jacobus~Hendricus van Lint, \emph{Graph theory,
  coding theory and block designs}, vol.~19, Cambridge University Press, 1975.

\bibitem{enomoto1979tight}
Hikoe Enomoto, Noboru Ito, and Ryuzaburo Noda, \emph{Tight 4-designs}, Osaka
  Journal of Mathematics \textbf{16} (1979), no.~1, 39--43.

\bibitem{godsil2001algebraic}
Chris Godsil and Gordon~F. Royle, \emph{Algebraic graph theory}, vol. 207,
  Springer Science \& Business Media, 2001.

\bibitem{higman1987coherent}
Donald~Gordon Higman, \emph{Coherent algebras}, Linear algebra and its
  applications \textbf{93} (1987), 209--239.

\bibitem{larman1977two}
David~G Larman, C~Ambrose Rogers, and Johan~J Seidel, \emph{On two-distance
  sets in euclidean space}, Bulletin of the London Mathematical Society
  \textbf{9} (1977), no.~3, 261--267.

\bibitem{lisonvek1997new}
Petr Lison{\v{e}}k, \emph{New maximal two-distance sets}, Journal of
  combinatorial theory, Series A \textbf{77} (1997), no.~2, 318--338.

\bibitem{nozaki2020maximal}
Hiroshi Nozaki and Masashi Shinohara, \emph{Maximal 2-distance sets containing
  the regular simplex}, Discrete Mathematics \textbf{343} (2020), no.~11,
  112071.

\bibitem{szollHosi2018constructions}
Ferenc Sz{\"o}ll{\H{o}}si and Patric~R.J. {\"O}sterg{\aa}rd,
  \emph{Constructions of maximum few-distance sets in euclidean spaces}, arXiv
  preprint arXiv:1804.06040 (2018).

\bibitem{xiang2018nonexistence}
Ziqing Xiang, \emph{Nonexistence of nontrivial tight 8-designs}, Journal of
  Algebraic Combinatorics \textbf{47} (2018), no.~2, 301--318.

\end{thebibliography}

\end{document}